\documentclass[a4paper,twoside,reqno]{amsart}

\usepackage{amssymb}
\usepackage[mathscr]{eucal}
\usepackage{hyperref}
\usepackage{enumitem}

\theoremstyle{plain}
\newtheorem{prop}{Proposition}[section]
\newtheorem{thm}[prop]{Theorem}
\newtheorem{cor}[prop]{Corollary}
\newtheorem{lem}[prop]{Lemma}

\theoremstyle{definition}

\newtheorem{exm}[prop]{Example}

\theoremstyle{remark}
\newtheorem{rem}[prop]{Remark}

\newcommand{\C}{{\mathbb{C}}}
\newcommand{\R}{{\mathbb{R}}}
\newcommand{\conv}{{\rm conv}}
\newcommand{\cone}{{\rm cone}}
\newcommand{\relint}{{\rm relint}}
\newcommand{\Aut}{{\rm Aut}}
\newcommand{\End}{{\rm End}}
\newcommand{\Mat}{{\rm Mat}}
\newcommand{\Sym}{{\rm Sym}}
\newcommand{\SO}{{\rm SO}}
\newcommand{\Ort}{{\rm O}}
\newcommand{\SL}{{\rm SL}}
\newcommand{\Sp}{{\rm Sp}}
\newcommand{\Ad}{{\rm Ad}}
\newcommand{\ad}{{\rm ad}}
\newcommand{\id}{{\rm id}}
\newcommand{\tr}{{\rm tr}}
\newcommand{\re}{{\rm Re}}
\newcommand{\im}{{\rm Im}}

\begin{document}

\title
[A Spectrahedral Representation for Polar Orbitopes]
{A Spectrahedral Representation for Polar Orbitopes}

\author
{Tim Kobert}

\address
 {Fachbereich Mathematik und Statistik \\
 Universit\"at Konstanz \\
 Germany}
\email
 {tim.kobert@uni-konstanz.de}

\begin{abstract}
Let $G$ be a Lie group with real semisimple Lie algebra $\mathfrak{g}$. Further let $\mathfrak{g} = \mathfrak{k} \oplus \mathfrak{p}$ be a Cartan decomposition. The maximal compact subgroup $K \subseteq G$ acts on $\mathfrak{p}$ via the adjoint representation and the convex hulls of the resulting orbits are the polar orbitopes. We prove that every polar orbitope is a spectrahedron by giving an explicit representation. In addition we give a new proof for the fact that the faces of a polar orbitope are, up to conjugation, given by the faces of the momentum polytope. 
\end{abstract}

\maketitle


\tableofcontents

\section*{Introduction}
A finite dimensional representation $\varphi : K \longrightarrow \Aut(V)$ of a group $K$ is \emph{polar} if $V$ is equipped with a $K$-invariant inner product and if there exists a subspace of $V$ which intersects every $K$-orbit of $V$ orthogonally. 
 A \emph{polar orbitope} is the convex hull of a $K$-orbit in a polar representation of a compact group $K$. Given a Lie group $G$  with real semisimple Lie algebra $\mathfrak{g}$, there is a Cartan decomposition $\mathfrak{g} = \mathfrak{k} \oplus \mathfrak{p}$, where $\mathfrak{k}$ is the $1$-eigenspace and $\mathfrak{p}$ the $-1$-eigenspace of the corresponding Cartan involution. Let $K$ be the analytical subgroup of $G$ which corresponds to the lie algebra $\mathfrak{k}$. Then $K$ is a maximal compact subgroup of $G$ and acts on $\mathfrak{p}$ via the adjoint representation. Such a representation is polar and every polar orbitope is isomorphic to the convex hull of such a $K$-orbit on $\mathfrak{p}$ for some real semisimple Lie algebra $\mathfrak{g}$ with Cartan decomposition $\mathfrak{g} = \mathfrak{k} \oplus \mathfrak{p}$ (see \cite{dad} Prop. 6). \\
In \cite{sss} R. Sanyal, F. Sottile and B. Sturmfels introduced the notion of orbitopes. They posed 10 questions about orbitopes and presented results regarding these questions. 
Some examples of polar orbitopes such as the symmetric and skew symmetric Schur-Horn orbitope were presented and some of the questioned were answered.
 In \cite{bgh1} and \cite{bgh2} L. Biliotti, A. Ghigi and P. Heinzner generalized parts of these results to polar orbitopes in general. 
 They proved that every face of a polar orbitope is exposed and gave a bijection between the face orbits of a polar orbitope and its momentum polytope (which is just the image of the given orbitope when orthogonally projected to a maximal abelian subspace of $\mathfrak{p}$).\\
Our main result is Theorem \ref{thmPolarIsSpectra} and is discussed in section \ref{secSpecProp}. The Theorem states that every polar orbitope is spectrahedral. In fact the theorem gives an explicit spectrahedral description of the orbitope. This representation as a spectrahedron generalizes \cite{sss} Theorem 3.4 for the symmetric Schur-Horn orbitope, Theorem 3.15 for the skewsymmetric Schur-Horn orbitope and Theorem 4.7 for the Fan orbitope.\\
 Up to $K$-conjugacy the momentum polytope of a polar orbitope characterizes the face structure of the polar orbitope completely. This is the main result in \cite{bgh2}.
In section \ref{secFaceCorresp} we present a new proof for this correspondence of face orbits.\medskip

\noindent {\bf Acknowledgement.} The author thanks C. Scheiderer for many illuminating conversations, critical comments and useful remarks.


\section{Spectrahedron property} \label{secSpecProp}

Throughout this article we will use the terminology that is used in \cite{kna}.
Let $G$ be a Lie group with real semisimple Lie algebra $\mathfrak{g}$.
We fix a Cartan decomposition $\mathfrak{g} = \mathfrak{k} \oplus \mathfrak{p}$ and the compact subgroup $K$ of $G$ with Lie algebra $\mathfrak{k}$.
The group $K$ acts on $\mathfrak{p}$ via the adjoint representation $\Ad : K \longrightarrow \Aut(\mathfrak{g})$. Let $x \in \mathfrak{p}$ then the polar orbitope $\mathcal{O}_x $ is given as the convex hull of the $K$-orbit of $x$, i.e. $\mathcal{O}_x := \conv(\Ad_K \cdot x)$.\\
An element $x \in \mathfrak{g}$ is {\it real semisimple} if $\ad_x \in \End(\mathfrak{g})$ is diagonalizable over $\R$. 
Let $\mathfrak{a}$ be a maximal abelian subspace of $\mathfrak{p}$. The Weyl group $W = W(\mathfrak{g}, \mathfrak{a})$ is the quotient of the normalizer $N_K (\mathfrak{a}) = \{ k \in K : \Ad_k (\mathfrak{a}) \subseteq \mathfrak{a}\}$ of $\mathfrak{a}$ in $K$ by the centralizer $Z_K(\mathfrak{a}) = \{ k \in K : \Ad_k | _\mathfrak{a} = \id _\mathfrak{a} \}$ of $\mathfrak{a}$ in $K$ and acts on $\mathfrak{a}$ via $\Ad$. Then the following was proved by B. Kostant:

\begin{prop}[\cite{kos} Prop. 2.4]\label{propKosKOrbitInA}
An element $x \in \mathfrak{g}$ is real semisimple iff it is conjugate to an element $a \in \mathfrak{a}$. In that case $(\Ad_G \cdot x) \cap \mathfrak{a} = (\Ad_K \cdot a) \cap \mathfrak{a} = W \cdot a$ is a $W$-orbit. 
\end{prop}

\noindent Since any element in $\mathfrak{p}$ is $K$-conjugate to an element in $\mathfrak{a}$ (\cite{kna} Thm. 6.51), the elements of $\mathfrak{p}$ are real semisimple.
Let $\mathfrak{g}^\C$ denote the complexification of $\mathfrak{g}$.

\begin{prop}\label{propsshasrealeigenvalues}
For any real semisimple $x \in \mathfrak{g}$ and any finite dimensional complex representation $\psi$ of $\mathfrak{g}^\C$ the endomorphism $\psi_x$ is diagonalizable and has real eigenvalues.
\end{prop}

\begin{proof}
Since $\ad_x$ is diagonalizable there is a Cartan subalgebra $\mathfrak{h}^\C $ of $\mathfrak{g}^\C$ such that $x \in \mathfrak{h}^\C$.
By \cite{kna} proposition 5.4 the endomorphism $\psi | _{\mathfrak{h}^\C}$ acts diagonalizable on $\mathfrak{p}$. Since $x$ is real semisimple the eigenvalues of $\ad_x$ are real. 
By \cite{sug} Theorem 2 there is an Iwasawa decomposition $\mathfrak{g} = \mathfrak{k}' \oplus \mathfrak{a}' \oplus \mathfrak{n}'$, such that $x \in \mathfrak{a}'$ and $\mathfrak{h}^\C \cap \mathfrak{g} = \mathfrak{a}' \oplus \mathfrak{t}'$ for some abelian subalgebra $\mathfrak{t}' \subseteq \mathfrak{k}'$. (In the notation of \cite{sug} the Cartan subalgebra $\mathfrak{h}^\C$ is standard w.r.t. the standard triple $(\mathfrak{k}', \mathfrak{p}', \mathfrak{a}')$, where $\mathfrak{g} = \mathfrak{k}' \oplus \mathfrak{p}'$ is a Cartan decomposition). Using \cite{kna} proposition 5.4 again, we see that the weights take real values on $\mathfrak{a}'$ and in particular on $x$.
\end{proof}

Combining proposition \ref{propsshasrealeigenvalues} with the fact that elements in $\mathfrak{p}$ are real semisimple, we can see that for any finite dimensional complex (meaning $V$ is a complex vectorspace) representation $\psi : \mathfrak{g}^\C \longrightarrow \End(V)$ and any $x \in \mathfrak{p}$ the eigenvalues of $\psi_x$ are real.

For any $x \in \mathfrak{p}$ let $\max \mu (\psi_x)$ denote the maximal eigenvalue of $\psi_x$. 
In order to show that $\mathcal{O}_x$ is a spectrahedron we will first proof the following theorem:

\begin{thm}\label{thm1}
For $x,y \in \mathfrak{p}$ the following are equivalent:
\begin{enumerate}
\item
$y \in \mathcal{O}_x$,
\item
$\max \mu (\varphi_y) \leq \max \mu (\varphi_x)$ for all fundamental representations $\varphi$ of $\mathfrak{g}^\C$.
\end{enumerate}
\end{thm}

\begin{rem}
The fundamental representations of the complexification $\mathfrak{g}^\C$ are building blocks for the irreducible representations. One can describe any irreducible finite dimensional complex representation of $\mathfrak{g}^\C$ as a Cartan product (see e.g. \cite{eas}) of finitely many fundamental representations.
The maximal eigenvalue of a Cartan product in a fixed point in $\mathfrak{p}$ is the sum of maximal eigenvalues of the factors.
The second condition in Theorem \ref{thm1} actually holds for all finite dimensional complex representations of $\mathfrak{g}^\C$. 
In order to find a spectrahedral representation though we use the fact that there are only finitely many fundamental representations.
\end{rem}

\begin{proof}
We will prove Theorem \ref{thm1} in 4 steps:\\
(1) \emph{$\Ad_K$-invariance:}
Obviously the first condition of Theorem \ref{thm1} is $\Ad_K$-invariant in $x$ as well as in $y$.
The same is true for the second condition: Let $\tilde{G}$ be the universal covering group of $G$. Let $\psi: \mathfrak{g} \longrightarrow \End(V)$ be a finite dimensional complex representation. There is a representation $\Psi: \tilde{G} \longrightarrow \Aut(V)$ with $\exp \circ \psi = \Psi \circ \exp$.
For $k \in K$, since $G = \tilde{G}/H$ for some subgroup $H$ of the centre $Z_{\tilde{G}}$, we can find an element $g \in \tilde{G}$ ($gH = k$) such that $\Ad_k \cdot x = \Ad_g \cdot x$. Then 
\[
\psi_{\Ad_k \cdot x} = \psi_{\Ad_g \cdot x} = \frac{\rm d}{{\rm d}t} ( \Psi_{\exp t \Ad_g \cdot x})|_{t=0} = \Psi_g \psi_x \Psi_g^{-1}.
\]
 So for any $x \in \mathfrak{p}$ the eigenvalues of $\psi_x$ and $\psi_{\Ad_k \cdot x}$ are the same.\medskip

\noindent(2) \emph{Highest weight gives largest eigenvalue:}
Since both conditions of Theorem \ref{thm1} are $\Ad_K$-invariant in $x$ as well as in $y$ and since every element in $\mathfrak{p}$ is $K$-conjugate to an element in $\mathfrak{a}$ we can assume that $x,y \in \mathfrak{a}$.
Let $\Sigma \subseteq \mathfrak{a}^*$ be the set of restricted roots.
We fix an ordering of $\mathfrak{a}^*$ and get a set $\Sigma^+$ of positive restricted roots and with that a simple system $\{\alpha_1,...,\alpha_k\}$ of the abstract root system $\Sigma$ of $\mathfrak{a}^*$ (\cite{kna} p. 118 and Cor. 6.53).
 The {\it Weyl chambers} are the connected components of 
\[
\mathfrak{a}\setminus \bigcup \limits _{\gamma \in \Sigma} \left\{ x \in \mathfrak{a} : \gamma (x) = 0 \right\}.
\]
By the $\Ad_K$-invariance and using the fact that the Weyl group acts transitively on the Weyl chambers (\cite{hel} Thm. 2.12 or implicitely \cite{kna} Thm. 2.63), we can choose $x$ and $y$ to be in the closed Weyl chamber $C = \{ z \in \mathfrak{a} : \alpha_1(z) \geq 0 ,..., \alpha_k(z) \geq 0\}$.\\
Let $\mathfrak{t}$ be a maximal abelian subspace of the centralizer $Z_\mathfrak{k}(\mathfrak{a}) = \{ k \in \mathfrak{k} : [k,\mathfrak{a}] = 0\}$. Then the complexification $\mathfrak{h}^\C$ of $\mathfrak{h} := \mathfrak{t} \oplus \mathfrak{a}$ is a Cartan subalgebra of $\mathfrak{g}^\C$. 
The restricted roots $\Sigma$ are given by the restrictions of the roots $\Delta$ of $(\mathfrak{g}^\C, \mathfrak{h}^\C)$ that are non-trivial on $\mathfrak{a}$ (\cite{kna} Prop. 6.47). 
The roots in $\Delta$ take real values on $i\mathfrak{t} \oplus \mathfrak{a}$.
The system $\Sigma^+$ of positive restricted roots can be extended to a system $\Delta^+$ of positive roots by choosing a lexicographic ordering on $(i\mathfrak{t} \oplus \mathfrak{a})^*$ which takes $\mathfrak{a}$ before $i\mathfrak{t}$.
The set of positive roots $\Delta^+$ determines a simple system $\{ \beta_1,...,\beta_l\}$ in $\Delta$ which is nonnegative on $C$. \\
Now if $\lambda$ is the highest weight of some representation $\psi: \mathfrak{g} \longrightarrow \End(V)$ then by the theorem of the highest weight (\cite{kna} Thm. 5.5) every weight of $\psi$ is of the form $\lambda - \sum _{i=1}^l n_i \beta_i$ with nonnegative integers $n_i$. Since the $\beta_i$ are by definition nonnegative on $C$ and since the weights are the eigenvalues of $\psi$, $\lambda(x)$ is the largest eigenvalue of $\psi_x$ for every $x \in C$.\medskip

\noindent (3) \emph{Reduction to the momentum polytope:}
In order to change the first condition of Theorem \ref{thm1} we use Kostant's convexity Theorem:

\begin{thm}[\cite{kos} Thm. 8.2]\label{thmKosConvexity}
Let $P: \mathfrak{p} \longrightarrow \mathfrak{a}$ be the orthogonal projection with respect to the inner product that is induced by the Killing form on $\mathfrak{g}$. Then for $x \in \mathfrak{a}$ we have:
\[
P(\Ad_K \cdot x) = \conv(W \cdot x).
\]
\end{thm}
 
Kostant's Theorem connects the polar orbitope $\mathcal{O}_x$ to the momentum polytope $\Pi_x := \conv (W \cdot x)$.  
 
\begin{lem}\label{lemReductionToA}
For $x \in \mathfrak{a}$ the equation $\mathcal{O}_x \cap \mathfrak{a} =  \Pi_x$ holds.
\end{lem}

\begin{proof}
By definition $W \cdot x \subseteq \Ad_K \cdot x$, so we have $\Pi_x \subseteq \conv(\Ad_K \cdot x) \cap \mathfrak{a}$. 
On the other hand let $y \in \mathcal{O}_x \cap \mathfrak{a}$. Then by Theorem \ref{thmKosConvexity} we have $y = P(y) \in P( \mathcal{O}_x) = \conv(P(\Ad_K \cdot x)) = \Pi_x$.
\end{proof}

\noindent(4) \emph{Final step:}
Let $x_1,...,x_n \in \mathfrak{a}$ be chosen such that $ \left< x_i, \cdot \right> = \alpha_i$, where $\left< \cdot, \cdot \right>$ denotes the inner product on $\mathfrak{a}$ induced by the Killing form.
Let $\lambda_1,...,\lambda_k \in \mathfrak{a}^*$ be the highest weights of the fundamental representations. 
The $\lambda_i$ are given by $2\left<\lambda_i, \alpha_j\right>_{\mathfrak{a}^*}/\left|\alpha_j\right|^2 = \delta_{ij}$ where $\left< \cdot, \cdot \right>_{\mathfrak{a}^*}$ denotes the inner product on $\mathfrak{a}^*$ that is induced by $\left< \cdot, \cdot \right>$ and the isomorphism $\mathfrak{a} \longrightarrow \mathfrak{a}^*, z \mapsto \left< z, \cdot \right>$. Then we have $\cone(x_i: i=1,...,k) = \{ z \in \mathfrak{a}: \lambda_1 (z) \geq 0,..., \lambda_k (z) \geq 0\}$. Kostant proved:

\begin{lem}[\cite{kos} Lem. 3.3]\label{lemKosKey}
Let $y \in C$, then $y \in \Pi_x$ iff $x-y \in \cone (x_i : i=1,...,k).$
\end{lem}

Combining Lemma \ref{lemReductionToA} and Lemma \ref{lemKosKey} we see that $y \in \mathcal{O}_x$ iff $\lambda_i(x-y) \geq 0$ for all $i=1,...,k$. Which is equivalent to $\lambda_i(x) \geq \lambda_i (y)$ for all $i = 1,...,k$. 
Since the highest weights give the largest eigenvalue of their representations when evaluated on $C$ this finishes the proof of Theorem \ref{thm1}.

\end{proof}

\begin{prop} \label{propHermitianInnerProduct}
Let $\mathfrak{g}$ be a real semisimple Lie algebra and let $\psi: \mathfrak{g}^\C \longrightarrow \End(V)$ be a finite dimensional complex representation. Then there is a hermitian inner product on $V$ such that for every $x \in \mathfrak{p}$ the endomorphism $\psi_x$ on $V$ is hermitian.
\end{prop}

\begin{proof}
The Lie algebra $\mathfrak{u} = \mathfrak{k} \oplus i \mathfrak{p}$ is the compact real form of $\mathfrak{g}^\C$ (\cite{kna} p. 292). 
Since $\mathfrak{u}$ is compact the complex conjugation with respect to $\mathfrak{u}$ is a Cartan involution on the real semisimple Lie algebra $(\mathfrak{g}^\C )^\R= \mathfrak{u} \oplus i\mathfrak{u}$.
Since $\mathfrak{u}$ is compact and semisimple the simply connected Lie group $U$ belonging to $\mathfrak{u}$ is compact.
Let $\Psi: U \longrightarrow \Aut(V)$ be the representation corresponding  to the restriction of $\psi$ to $\mathfrak{u}$.
Further let $\left< \cdot , \cdot \right>_V$ be any hermitian inner product on $V$. Using the Haar measure we can define a $U$-invariant inner product on $V$:
\[
\left< v,w \right> _U := \int_{u \in U} \left< \Psi_u v, \Psi_u w \right>_V du.
\]
Let $x \in \mathfrak{p}$. Differentiating the equation $\left< \Psi_{e^{tix}} v, \Psi_{e^{tix}} w \right>_{U} = \left< v, w \right>_{U}$ and evaluating in $t=0$ we get $
\left< \psi_{ix} v, w \right>_{U} + \left< v, \psi_{ix} w \right>_{U} = 0,$
which is equivalent to $\left< \psi_{x} v, w \right>_{U} = \left< v, \psi_{x} w \right>_{U}$. So $\psi_x$ is self-adjoint.
\end{proof}

Using Theorem \ref{thm1} and interpreting $\varphi_z$ (for representations $\varphi : \mathfrak{g}^\C \longrightarrow \End(V)$ and $z \in \mathfrak{p})$ as matrices by fixing a basis and taking the inner product given by Proposition \ref{propHermitianInnerProduct} we get:

\begin{thm}\label{thmPolarIsSpectra}
For every $x \in \mathfrak{p}$ the polar orbitope $\mathcal{O}_x$ is a spectrahedron given as
\[
\left\{ y \in \mathfrak{p}: \max \lambda(\varphi_x) I_{\dim \varphi} - \varphi_y \succeq 0 \hspace{5pt}{\rm for \hspace{3pt} all \hspace{3pt} fundamental \hspace{3pt} representations}\hspace{3pt} \varphi \hspace{3pt}{\rm of}\hspace{3pt} \mathfrak{g}^\mathbb{C} \right\}.
\]
\end{thm}

\begin{rem}
The matrices $\varphi_y$ in Theorem \ref{thmPolarIsSpectra} might have nonreal entries but there is an easy way to convert this representation of $\mathcal{O}_x$ into a representation with real matrices of twice the size. Let $S = \{ x \in \R^n : A(x) \succeq 0\}$, where $A(x) = A_0 + x_1 A_1 + ... + x_n A_n$ with hermitian matrices $A_i$. Then $S$ is also given as
\[
S = \left\{ x \in \R^n : 
\left(\begin{array}{cc}
\re(A(x)) & \im(A(x)) \\
-\im(A(x)) & \re(A(x))
\end{array}\right)
 \succeq 0\right\}.
\]
\end{rem}

Since every face of a spectrahedron is exposed, Theorem \ref{thmPolarIsSpectra} implies that every face of a polar orbitope is exposed. This fact was proved in a different way in \cite{bgh2}.
Theorem \ref{thmPolarIsSpectra} is a generalization of the corresponding results in \cite{sss} about the symmetric and skew symmetric Schur-Horn orbitopes and the Fan orbitope. The setup in \cite{sss} is slightly different but the results are implied by Theorem \ref{thmPolarIsSpectra}.

\begin{rem}
\emph{Schur-Horn orbitopes:}
(1) In order to see how Theorem \ref{thmPolarIsSpectra} generalizes \cite{sss} Thm. 3.4 for the symmetric Schur-Horn orbitope, we study the real semisimple Lie algebra $\mathfrak{g} = \mathfrak{sl}_n (\R)$ = $\Mat_n(\R) \cap \{ \tr = 0 \}$.
 Decomposing $\mathfrak{g}$ into the skew symmetric and the symmetric part gives a Cartan decomposition. The special orthogonal group $\SO(n)$ acts on the symmetric matrices in $\mathfrak{sl}_n(\R)$ by conjugation (adjoint representation).
In \cite{sss} section 3.1 the orthogonal group $\Ort(n)$, acting by conjugation on the set of real symmetric matrices $\Sym_n(\R)$ (of any trace), is considered instead. 
Let $M \in \Sym_n(\R)$ then $M_0 := M - \frac{\tr(M)}{n} I_n$ has trace $0$ and the orbits $\Ort(n) \cdot M_0$ and $\SO(n) \cdot M_0$ coincide. 
By Theorem \ref{thmPolarIsSpectra} the orbitope $\conv(\SO(n) \cdot M_0)$ is a spectrahedron. Since $\conv(\Ort(n) \cdot M)$ is just a translation by $\frac{\tr(M)}{n} I_n$ it is a spectrahedron as well.\medskip

\noindent(2) Next we consider the Lie algebra $\mathfrak{g} = \mathfrak{so}_n(\C)$ as a real semisimple Lie algebra with Cartan  decomposition $\mathfrak{g} =  \mathfrak{so}_n(\R) \oplus i \hspace{1pt} \mathfrak{so}_n(\R)$ (where $i := \sqrt{-1}$). Again $\SO(n)$ is the acting group and acts on $i \hspace{1pt} \mathfrak{so}_n(\R)$ by conjugation. The skewsymmetric Schur-Horn orbitope is the convex hull of such an orbit.
The authors of \cite{sss} consider the group $\Ort(n)$ instead.
If $n$ is even it makes a difference whether the acting group is $\SO(n)$ or $\Ort(n)$. But with a similar trick as for the Fan orbitope (next example), by adding a $0$ column and row, one can see that Theorem \ref{thmPolarIsSpectra} implies \cite{sss} Theorem 3.15.
\end{rem}

The semisimple Lie groups are locally direct products of simple Lie groups. There exist complete tables of simple Lie groups and a lot of their properties are listed (see e.g. \cite{sug} or \cite{tit}). We use these tables (and \cite{fh}) to discuss two examples in more detail:

\begin{exm}
\emph{Fan orbitope (see also \cite{sss} Section 4.3) :}
We take a look at the real semisimple Lie algebra $\mathfrak{g} = \mathfrak{so}_{m,n} = \{ X \in \mathfrak{gl}_{m+n}(\mathbb{R}) : X^T I_{m,n} + I_{m,n} X = 0 \}$, where 
\[
I_{m,n} :=  \left( \begin{array}{cc} I_m & 0 \\ 0 & -I_n \end{array} \right) 
\]
and $m+n \geq 3$. The Lie algebra $\mathfrak{so}_{m,n}(\R)$ may also be given as
\[
\left\{ \left( 
\begin{array}{cc} 
A & B \\
B^T & C
\end{array} \right) \in \mathfrak{gl}_{m+n}(\mathbb{R}) : A \in \mathfrak{so}_m (\R), C  \in \mathfrak{so}_n (\R), B \in \Mat_{m,n}(\R)  \right\}.
\]
The involution $\theta: X \mapsto -X^T$ is a Cartan involution on $\mathfrak{so}_{m,n}(\R)$ and the corresponding Cartan decomposition is
\[
\hspace{19pt} \mathfrak{so}_{m,n}(\R) = \mathfrak{k} \oplus \mathfrak{p} = \left\{\left( 
\begin{array}{cc} 
A & 0 \\
0 & C
\end{array} \right)
 \in \mathfrak{so}_{m,n}(\R)\right\} \oplus \left\{\left( 
\begin{array}{cc} 
0 & B \\
B^T & 0
\end{array} \right)
 \in \mathfrak{so}_{m,n}(\R)\right\}.
\]
One choice for the maximal abelian subspace $\mathfrak{a}$ of $\mathfrak{p}$ is 
\[
\mathfrak{a} =  \left\{\left( \begin{array}{cc} 
0 & D \\
D^T & 0
\end{array} \right)
 \in \mathfrak{so}_{m,n}(\R): D \in \Mat_{m,n}(\R), D = (d_{pq})_{p,q}, d_{pq} = 0 \text{ for } p \neq q  \right\}.
\]
The maximal compact subgroup of $G$ is $K = \SO(m) \times \SO(n)$. 
The complexification $\mathfrak{so}_{m,n}(\R) ^\C$ is isomorphic to $ \mathfrak{so}_{m+n}(\C)$ via conjugation by 
\[
J = \left( \begin{array}{cc} 
iI_m & 0 \\
0 & I_n
\end{array} \right), \hspace{3pt }{\rm i.e.} \hspace{3pt}
X = \left( \begin{array}{cc} 
A & B \\
B^T & C
\end{array} \right) 
\mapsto JXJ^{-1} = 
\left( \begin{array}{cc} 
A & iB \\
-iB^T & C
\end{array} \right) ,
\]
where $i := \sqrt{-1}$. What the fundamental representations of $\mathfrak{so}_{m+n}(\C)$ look like depends on whether $m+n$ is odd or even.
In the odd case $\mathfrak{so}_{m+n}(\C)$ is of type $B_r$, where $2r+1 = m+n$. The fundamental representations are the exterior powers $\varphi_p = \bigwedge^p \varphi_1 : \mathfrak{g}^\C \longrightarrow \End (\bigwedge ^p \C^{m+n})$ of the natural representation $\varphi_1 : \mathfrak{g}^\C \longrightarrow \End(\C^{m+n}), M \mapsto (v \mapsto Mv)$ for $p=1,...,r-1$ and the spin representation $\varphi_{r}$. 
The corresponding highest weights are given by $L_1+...+L_p$ for $p \leq r-1$ and $(L_1+...+L_r)/2$, where 
\[
L_j : \mathfrak{a} \longrightarrow \R, \left( 
\begin{array}{cc} 
0 & D \\
D^T & 0
\end{array} \right) \mapsto (e_j^n )^T D e_j^m
\]
maps to the $j$-th diagonal entry of $D$ for $j \leq \min(m,n)$ and maps to $0$ for $j > \min(m,n)$. So the maximal eigenvalue of a representation $\varphi_p$ when evaluated in an element $x = {0 \hspace{3pt} B \choose B^T 0} \in \mathfrak{p}$
is the sum (or in the case $p=r$ half the sum) of the $p$ largest singular values of $B$ (observe that $B$ is not a square matrix, so all the singular values will be achieved). Since the highest weight only differs by the factor 2, we can use the $r$-th exterior power instead of the spin representation. Assume $m>n$ and let $s_1\geq ...\geq s_n\geq 0$ be the singular values of $B$. Then the polar orbitope $\mathcal{O}_x$ is given as 
\begin{equation}\tag{1.1}\label{SpecRep}
\left\{ y \in \mathfrak{p}: (s_1+...+s_p) I_{\dim \varphi_p} - \left(\bigwedge\nolimits^p\varphi_1\right)_y \succeq 0 \hspace{5pt}{\rm for \hspace{3pt} all} \hspace{3pt} p=1,...,r \right\}.
\end{equation}\\
In case $n+m$ is even $\mathfrak{so}_{m+n}(\C)$ is of type $D_r$, where $2r = n+m$. 
The first $r-2$ fundamental representations $\varphi_1,...,\varphi_{r-2}$ are again the exterior powers of the natural representation. The two other fundamental representations $\varphi_{r-1}, \varphi_r$ are the halfspin representations.
 The highest weights of the halfspin representations are $(L_1 + ... + L_{r-1} - L_r)/2$ and $(L_1 + ... + L_{r})/2$.
The maximal eigenvalues of the fundamental representations on $\mathfrak{p}$ are again given as sums (and in one case a difference) of the singular values of the $B$ part. The spectrahedral representation looks very similar to (\ref{SpecRep}), but one has to differentiate between the exterior powers and the halfspin representations.\\
In \cite{sss} the authors study the group $\Ort(n) \times \Ort(n)$ acting on $\Mat_{n}(\R)$ by $(g,h) \cdot B = gBh^T$ which comes down to the adjoint action (conjugation) of ${g \hspace{3pt} 0 \choose 0 \hspace{3pt} h}$ on ${0 \hspace{3pt} B \choose B^T 0}$. 
For every $B \in \Mat_{n}(\R)$ there exist $g,h \in \Ort(n)$ such that $gBh^T$ is diagonal, with entries being the singular values of $B$ (singular value decomposition).
Let $G = \Ort(n) \times \Ort(n)$ and $H = \SO(n) \times \SO(n)$, both groups act on $\Mat_n(\R)$ via $(g,h) \cdot B = gBh^T$.
If $B$ is regular, the orbits $G \cdot B$ and $H \cdot B$ are different.
But we can embed $\Mat_n(\R)$ into $\Mat_{2n+1}(\R)$ via
\[B \mapsto B_0 := 
\left( 
\begin{array}{cc|c} 
0 & 0 & B \\
0 & 0 & 0 \\ \hline \rule{0pt}{2.6ex}
B^T & 0 & 0
\end{array} \right) \in \mathfrak{p} \subseteq \mathfrak{so}_{n+1,n} \subseteq \Mat_{2n+1}(\R),
\]
where $\mathfrak{so}_{n+1,n} = \mathfrak{k} \oplus \mathfrak{p}$ as above.
Let $B,M \in \Mat_{n}(\R)$. We write $C_M := \conv \left(G \cdot M\right)$ and $\mathcal{O}_{M_0} := \conv \left(K \cdot M_0\right)$, where $K := \SO(n+1) \times \SO(n)$ acts on $\mathfrak{p}$ by conjugation.
Without loss of generality we assume that $M$ is diagonal.
Using the embedding $B \mapsto B_0$, we claim that $C_M = \mathcal{O}_{M_0} \cap \Mat_n (\R)$.\\
Let $B$ be in the orbit $G \cdot M$. Then $B$ has the same singular values as $M$ and $B_0$ has the same eigenvalues as $M_0$. Due to the $0$-row/column $B_0$ having the same eigenvalue as $M_0$ is equivalent to $B_0 \in K \cdot M_0$. So $B \in \mathcal{O}_{M_0} \cap \Mat_n (\R)$ for every $B \in G \cdot M$. Since $\mathcal{O}_{M_0} \cap \Mat_n (\R)$ is convex we have $C_M \subseteq \mathcal{O}_{M_0} \cap \Mat_n (\R)$.\\
On the other hand let $B_0 \in \mathcal{O}_{M_0} \cap \Mat_n (\R)$. 
Then $B_0$ is $K$-conjugate to an element $D_0 \in \mathcal{O}_{M_0} \cap \Mat_n (\R)$, where $D \in \Mat_n(\R)$ is diagonal (\cite{kna} Thm. 6.51 with $\mathfrak{a}$ as above).
The matrices $B_0$ and $D_0$ have the same eigenvalues and by Kostant's convexity Theorem \ref{thmKosConvexity} we have $D_0 = P(D_0) \in \conv( W \cdot M_0)$. 
The Weyl group $W$ acts on $M_0$ by permuting and changing signs (due to the $0$-row/column not just even sign changes) of the diagonal entries of $M$.
So $W \cdot M_0 \subseteq C_M$ and we have $D_0 \in \conv( W \cdot M_0) \subseteq C_M$.
Since $D_0$ and $B_0$ have the same eigenvalues, $D$ and $B$ have the same singular values, so $B \in C_M$.
We see that $C_M = \mathcal{O}_{M_0} \cap \Mat_{n}(\R)$ is a spectrahedron.\medskip
\end{exm}

\begin{exm}
\noindent Next we consider the Lie algebra 
\[
\mathfrak{g} = \mathfrak{sl}_{n+1}(\mathbb{H}) =
 \left\{ \left( 
\begin{array}{cc} 
A & B \\
-\overline{B} & \overline{A}
\end{array} \right) \in \mathfrak{gl}_{{n+1}}(\C) : A, B \in \mathfrak{gl}_m(\C),  \re(\tr(A)) = 0 \right\},
\]
where $n+1 = 2m \geq 4$ is an even integer.
As Cartan involution we choose $\theta$ to be the negative conjugate transpose, i.e. $\theta(X) = -\overline{X}^T$. The resulting Cartan decomposition is $\mathfrak{sl}_{n+1}(\mathbb{H}) = \mathfrak{k} \oplus \mathfrak{p}$, where
\[
\mathfrak{k} = \left\{\left( 
\begin{array}{cc} 
A & B \\
-\overline{B} & \overline{A}
\end{array} \right) \in \mathfrak{sl}_{n+1}(\mathbb{H}): A = -\overline{A}^T, B = B^T \right\} ,
\]
and
\[
\mathfrak{p}=
\left\{\left( 
\begin{array}{cc} 
A & B \\
-\overline{B} & \overline{A}
\end{array} \right) \in \mathfrak{sl}_{n+1}(\mathbb{H}): A = \overline{A}^T, \tr(A) = 0, B = -B^T \right\} .
\]
The real diagonal matrices
\[
\mathfrak{a} = \left\{\left( 
\begin{array}{cc} 
D & 0 \\
0 & D
\end{array} \right) \in \mathfrak{gl}_{n+1}(\R): \tr (D) = 0, D = (d_{pq})_{p,q}, d_{pq} = 0 \hspace{3pt} {\rm for} \hspace{3pt} p \neq q \right\}.
\]
form a maximal abelian subspace of $\mathfrak{p}$. The maximal compact subgroup $K$ of $G = \SL_{m}(\mathbb{H})$ is the compact symplectic group $\Sp(m)$.
The complexification of $\mathfrak{sl}_{n+1}(\mathbb{H})$ is $\mathfrak{sl}_{n+1}(\C)$. Its fundamental representations are the exterior powers $\varphi _p = \bigwedge ^p \varphi_1$ of the natural representation $\varphi_1$, for $p=1,...,n$. 
Let $x \in \mathfrak{p}$ then the eigenvalues of $(\varphi_i)_x$ are just the sums of $p$ eigenvalues of $x$. The Weyl group acts on $\mathfrak{a}$ as permutations of the diagonal entries of $D$.
 Taking $s_1 \geq ... \geq s_{r}$ to be the eigenvalues of $x$ (where $r = n$), we get representation (\ref{SpecRep}) for $\mathcal{O}_x$.
 \end{exm}


\section{Correspondence of face orbits}\label{secFaceCorresp}
We use the same notation as in section \ref{secSpecProp}. The Lie group $K$ acts on the set of all faces of the polar orbitope $\mathcal{O}_x$ by $k \cdot F = \{ \Ad_k \cdot y : y \in F \}$. 
Assuming $x \in \mathfrak{a}$ (for easier notation) another orbitope closely related to $\mathcal{O}_x$ is the momentum polytope $\Pi_x  := \conv(W \cdot x) = P(\mathcal{O}_x) \subseteq \mathfrak{a}$ (see Thm. \ref{thmKosConvexity}).
The Weyl group $W$ acts on $\mathfrak{a}$ and again by pointwise operation gives an action on the set of faces of $\Pi_x$.
The main result in \cite{bgh2} was to give a bijection between the $K$-orbits of faces of $\mathcal{O}_x$ and the $W$-orbits of faces of $\Pi_x$. We will present a new proof for this correspondence. The proof makes use of the fact that all the faces of $\mathcal{O}_x$ are exposed and of Kostant's convexity Theorem. It is partly based on the proof of Theorem 3.5 in \cite{sss}.

\begin{thm}\label{thmFaceCorresp}
Let $x \in \mathfrak{a}$ and let $F \subseteq \mathcal{O}_x$ be a face. Then we have:
\begin{enumerate}
\item \label{thmFaceCorresp1}
There exists $k \in K$ such that $f := P(k \cdot F)$ is a face of $\Pi_x$.
\item \label{thmFaceCorresp2}
Let $f$ and $k$ be as in (\ref{thmFaceCorresp1}), then $k \cdot F = P^{-1}(f) \cap \mathcal{O}_x$ and $f = (k \cdot F) \cap \mathfrak{a}$.
\item\label{thmFaceCorresp3}
Let $f \subseteq \Pi_x$ be a face and $\sigma \in W$. Further let $F := P^{-1}(f)\cap \mathcal{O}_x$ and $F' := P^{-1}(\sigma f)\cap \mathcal{O}_x$, then there exists $k \in K$ such that $F' = k \cdot F$.
\item\label{thmFaceCorresp4}
Let $f, f' \subseteq \Pi_x$ be faces such that $F := P^{-1}(f)\cap \mathcal{O}_x$ and $F' := P^{-1}(f')\cap \mathcal{O}_x$ satisfy $F' = k \cdot F$ for some $k \in K$. Then $f$ and $f'$ are on the same $W$-orbit. 

\end{enumerate}
\end{thm}

Let $V$ be a finite dimensional vector space. We write $H_{l,\alpha} \subseteq V$ for the hyperplane $\{ x \in V : l(x) = \alpha\}$, where $l \in V^*$ and $\alpha \in \R$. 
A hyperplane $H_{l,\alpha}$ is a { \it supporting hyperplane} of a convex set $C \subseteq V$ if $l(C) \geq 0$ and there exists $c \in C$ such that $l(c) = 0$.
Let $l \in \mathfrak{a}^*$ and let $a \in \mathfrak{a}$ such that $l = \left< a, \cdot \right>$. Since $\mathfrak{a} \subseteq \mathfrak{p}$ this also gives a linear functional on $\mathfrak{p}$. In this way we will interpret the elements of $\mathfrak{a}^*$ as elements in $\mathfrak{p}^*$

\begin{rem}\label{remLyEqLPy}
Let $l \in \mathfrak{a}^*$ and $y \in \mathfrak{p}$. Since $P$ is the orthogonal projection onto $\mathfrak{a}$ we have $l(y) = l(P(y))$.
\end{rem}

\begin{proof}
Without loss of generality we assume that $F$ is a proper face.\medskip

\noindent (\ref{thmFaceCorresp1}) Since $\mathcal{O}_x$ is a spectrahedron $F$ is exposed, so there exists a supporting hyperplane $H_{l,\alpha}$ of $\mathcal{O}_x$ with linear functional $l \in \mathfrak{p}^*$ and $\alpha \in \R$ such that $F = \mathcal{O}_x \cap H_{l,\alpha}$.
Then there exists $A \in \mathfrak{p}$ with $l = \left< A, \cdot \right>$. Since every element in $\mathfrak{p}$ is $K$-conjugate to an element in $\mathfrak{a}$ there exists $k \in K$ such that $\Ad_k \cdot A \in \mathfrak{a}$. 
Then the linear functional $l^k := \left< \Ad_k \cdot A, \cdot \right>$ lies in $\mathfrak{a}^*$ as well as in $\mathfrak{p}^*$. 
Since the Killing form is $\Ad_k$-invariant we have $F^k := \mathcal{O}_x \cap H_{l^k,\alpha} = k \cdot F$.
By Remark \ref{remLyEqLPy} we have $l^k(y) = l^k(P(y))$ for all $y \in \mathfrak{p}$.
So by Kostant's convexity Theorem $P(F^k) = P(\mathcal{O}_x \cap H_{l^k, \alpha}) = \Pi_x \cap H_{l^k, \alpha}$ and since $\Pi_x \subseteq \mathcal{O}_x$ and $F^k \neq \emptyset$ the hyperplane $H_{l^k,\alpha}$ is a supporting hyperplane for $\Pi_x$. So $f := P(F^k)$ is a face of $\Pi_x$. \medskip

\noindent (\ref{thmFaceCorresp2}) Clearly we have $F^k \subseteq P^{-1}(f) \cap \mathcal{O}_x$. On the other hand let $y \in P^{-1}(f) \cap \mathcal{O}_x = P^{-1}(\Pi_x \cap H_{l, \alpha}) \cap \mathcal{O}_x$. By Remark \ref{remLyEqLPy} we have $l (y) = l(P(y)) = \alpha$, so $y \in \mathcal{O}_x \cap H_{l, \alpha} = F^k$.
Using Lemma \ref{lemReductionToA} we can see that the equality $f = \Pi_x \cap H_{l, \alpha} = \mathcal{O}_x \cap \mathfrak{a} \cap H_{l, \alpha} = F^k \cap \mathfrak{a}$ holds. \medskip

\noindent (\ref{thmFaceCorresp3}) Let $k \in N_K(\mathfrak{a})$ be a representative of $\sigma$ (by definition $W = N_K(\mathfrak{a}) / Z_K(\mathfrak{a})$).  Since $\Pi_x$ is a polytope $f$ is an exposed face. Let $f = H_{l,\alpha} \cap \Pi_x$ and $l(\Pi_x) \geq 0$, where $l = \left< A, \cdot \right>$, $A \in \mathfrak{a}$. Define $l^k := \left< \Ad_k \cdot A, \cdot \right>$, then $\sigma f = H_{l^k,\alpha} \cap \Pi_x$. We conclude $F' = H_{l^k, \alpha} \cap \mathcal{O}_x = k \cdot \left( H_{l,\alpha} \cap \mathcal{O}_x\right) = k \cdot F$.
 
 Before we move to proving (\ref{thmFaceCorresp4}) we prove the following lemma:

\begin{lem}\label{lemFindSigma}
Let $f \subseteq \Pi_x$ be a nonempty face and let $y \in f, z \in \Pi_x$. Let $y \in \Pi_z$, then there exists $\sigma \in W$ so that $\sigma z \in f$. 
\end{lem}

\begin{proof}
We can assume that $f$ is a proper face of $\Pi_x$. So there is a supporting hyperplane $H \subseteq \mathfrak{a}$ of $\Pi_x$ such that $f = \Pi_x \cap H$. 
Since $y \in \Pi_z \cap H$ and $\Pi_z \subseteq \Pi_x$ the hyperplane $H$ is also a supporting hyperplane for $\Pi_z$.
 So $f' := \Pi_z \cap H$ is a face of $\Pi_z$. The face $f'$ contains a vertex $z'$ of $\Pi_z$. So there is a $\sigma \in W$ with $\sigma z = z'$. And we have $z' \in f' = \Pi_z \cap H \subseteq \Pi_x \cap H = f$.
\end{proof}

\noindent (\ref{thmFaceCorresp4}) Let $f = H_{l,\alpha} \cap \Pi_x$ with $l\in \mathfrak{a}^*$ and $\alpha \in \R$. By Remark \ref{remLyEqLPy} we have $F = H_{l,\alpha} \cap \mathcal{O}_x$ and using (\ref{thmFaceCorresp2}) we see that $f = F \cap \mathfrak{a}$. 
Let $z$ be in the relative interior of $f$, i.e. $z \in\relint(f)$. Since $f = F \cap \mathfrak{a}$ we have $z \in F$, so $\Ad_k \cdot z \in F'$ and $y := P(\Ad_k \cdot z) \in f'$.
By Kostant's convexity Theorem (Thm. \ref{thmKosConvexity}) we have $y \in \Pi_z$. By Lemma \ref{lemFindSigma} there is a $\sigma \in W$ with $\sigma z \in f'$. Since $z \in \relint(f)$ we have $\sigma z \in \relint(\sigma f)$.
So $\sigma f$ is the smallest face of $\Pi_x$ containing $\sigma z$. Since $\sigma z \in f'$ we see that $\sigma f \subseteq f'$ and $\dim (f) = \dim(\sigma f) \leq \dim f'$. We can use the same arguments to show $\dim (f') \leq \dim (f) = \dim (\sigma f)$, so $\sigma f = f'$.
\end{proof}

\begin{cor}
There is a bijection between the $K$-orbits of faces of $\mathcal{O}_x$ and the $W$-orbits of faces of $\Pi_x$. Let $f$ be a face of $\Pi_x$, then the bijection maps the $W$-orbit represented by $f$ to the $K$-orbit represented by $P^{-1}(f) \cap \mathcal{O}_x$.
\end{cor}

\begin{proof}
Theroem \ref{thmFaceCorresp} gives an inverse mapping and shows that these mappings are well defined.
\end{proof}

As an immediate consequence there exist only finitely many $K$-orbits of faces of the polar orbitope $\mathcal{O}_x$.


\begin{thebibliography}{}

\bibitem{bgh1}
L.~Biliotti, A.~Ghigi, P.~Heinzner:
\emph{Coadjoint Orbitopes} 
Osaka J. Math. Vol. 51, No. 4 (2014), p. 935-969.

\bibitem{bgh2}
L.~Biliotti, A.~Ghigi, P.~Heinzner:
\emph{Polar Orbitopes}.
Communications in Analysis and Geometry, Vol. 21, No. 3 (2013), p. 579-606.

\bibitem{dad}
J. ~Dadok:
\emph{Polar coordinates induced by actions of compact Lie groups}.
Trans. Amer. Math. Soc., Vol. 288, No. 1 (1985): p. 125-137.

\bibitem{eas}
M. ~Eastwood:
\emph{ The Cartan Product}.
Bull. Belg. Math. Soc. Simon Stevin, Vol. 11, No. 5 (2005), p. 641-651. http://projecteuclid.org/euclid.bbms/1110205624.

\bibitem{fh}
W. ~Fulton, J. Harris:
\emph{Representation Theory : A First Course}. 
Springer-Verlag, Berlin, Heidelberg, New York, 1991.

\bibitem{hel}
S. ~Helgason:
\emph{Differential Geometry, Lie Groups, and Symmetric Spaces}.
Academic Press, New York, 1978.

\bibitem{kna}
A.~W.~Knapp:
\emph{Lie Groups Beyond an Introduction}. Progress in Mathematics Vol. 140 (1996), Birkh\"auser.

\bibitem{kos}
B. ~Kostant:
\emph{On Convexity, the Weyl group and the Iwasawa decomposition}. Annales scientifiques de l'\'E.N.S. 4e s\'erie, tome, Vol. 6, No. 4 (1973), p. 413 - 455.

\bibitem{sss}
R.~Sanyal, F.~Sottile, B.~Sturmfels:
\emph{Orbitopes}.
Mathematika, Vol. 57, No. 2 (2011), p. 257-314.

\bibitem{sug}
M. ~Sugiura:
\emph{Conjugate classes of Cartan subalgebras in real semisimple Lie algebras}. 
Journal of the Mathematical Society of Japan, Vol. 11, No. 4 (1959).

\bibitem{tit}
J. ~Tits:
\emph{Tabellen zu den einfachen Lie Gruppen und ihren Darstellungen}.
Springer Verlag, Berlin, Heidelberg, New York, 1967.

\end{thebibliography}
\end{document}